\documentclass[a4paper]{amsart}
\usepackage{amssymb,amsmath,amsthm,latexsym}
\usepackage{footnote}
\usepackage{xcolor}
\usepackage{graphicx}
\usepackage{float}
\usepackage[margin=1in]{geometry}
\usepackage[width=1\textwidth]{caption}
\usepackage[para]{threeparttable}
\usepackage{color,soul}
\usepackage{tikz}
\usepackage{endnotes}
\usepackage[toc]{appendix}
\usetikzlibrary{arrows}

\newtheorem{thm}{Theorem}[section]
\newtheorem{cor}[thm]{Corollary}
\newtheorem{lem}[thm]{Lemma}

\newtheorem{defi}[thm]{Definition}
\newtheorem{examp}[thm]{Example}

{ \theoremstyle{remark}\newtheorem{remark}{Remark} }

\begin{document}

\title{The Second Generalization of the Hausdorff Dimension Theorem for Random Fractals}
\date{September 5, 2021}
\maketitle
\begin{center}
\author {Mohsen Soltanifar\footnote{Biostatistics Division, Dalla Lana School of Public Health, University of Toronto, Toronto, ON, Canada\\
e-mail: mohsen.soltanifar@alum.utoronto.ca \\
ORCID: https://orcid.org/0000-0002-5989-0082}}
\end{center}

\begin{abstract}
In this paper, we present a second partial solution for the problem of cardinality calculation of the set of fractals for its subcategory of the random virtual ones. Consistent with the deterministic case, we show that for the given quantities of Hausdorff dimension and Lebesgue measure, there are aleph-two virtual random fractals with almost surely Hausdorff dimension of a bivariate function of them and the expected Lebesgue measure equal to the later one. The associated results for three other fractal dimensions are similar to the case given for the Hausdorff dimension. The problem remains unsolved for the case of non-Euclidean abstract fractal spaces.
\end{abstract}
																												

 
\section{Introduction}

\subsection{Random Fractals}
Random fractals have emerged as the natural expansions of the deterministic fractals and their introductory mathematical treatment started with the works of Mandelbrot in early 1970s \cite{I1,I2}. Later, their rigorous mathematical treatment were solidified with the works of Taylor, Falconer and Graf in mid 1980s \cite{I3,I4,I5}. These fractals use stochastic processes rather than deterministic processes in their constructions and are characterized by statistical self-similarity rather than the deterministic precise self-similarity. In details, the small component parts of the fractal have the same probability distribution as the whole fractal.  Some of their applications have emerged in financial markets, cosmogeny and image synthesis. As in their deterministic counterparts, random fractals are divided in two categories: the real random fractals and the virtual random fractals. Some examples of the real word random fractals include the human brain surface and the coastlines of the British Isles. Others in the virtual random fractals category include random Cantor sets, random von Koch curves and Galton-Watson fractals.

\subsection{Existence}
Existence of abstract mathematical objects and its mathematical philosophy importance have long been the subject of interest among mathematicians and the first prominent argument for them roots to the work of Frege in early 1950\cite{I6}. Furthermore, it has long been discussed among mathematicians that whether existence of mathematical objects implies the possibility of their constructions,  and whether there is systematic approach to construction itself \cite{I7,I8,I9}. Also, applications of the constructive mathematical objects in many branches of mathematics and computer science add to the prominence of their investigations \cite{I10,I11}. In the case of fractals, this prominence increases given their applications in many scientific fields. Sets in fractal geometry are often characterized by their sizes measured by the fractal dimensions (e.g.,  Hausdorff dimension) and the space measures (e.g., Lebesgue measure). Hence, the investigation of existence of fractals is formulated in terms of their fractal dimension and the associated space measure. 

\subsection{Motivation}
Till mid 2000s, literature in fractal geometry has mostly been focused in calculating fractal dimensions and space measures of the discussed objects without retrospective investigation of their existence for these two given quantities. The earliest treatment of the problem of retrospective existence of fractals in terms of their fractal dimension have been through the works of Sharapovs and Soltanifar \cite{I12,I13}. The later work has expanded the former one and presented such existence as the Hausdorff Dimension Theorem (HDT) limited to deterministic thin virtual fractals with provision of cardinality of continuum. Similar results to the HDT have been reported in late 2010s in the works of Squillace and Gryszka \cite{I14,I15}. The HDT has recently been generalized to deterministic fat virtual fractals with provision of cardinality of aleph-two coined as the Generalized Hausdorff Dimension Theorem (GHDT), \cite{I16}. However, no information is available on the counterpart existence statement for the case of random fractals. Moreover, the current results in the deterministic realm are limited to the Hausdorff dimension and their validity is unclear for other fractal dimensions. Finally, there are still open questions whether the cardinality of the set of fractals depends to their deterministic or random nature  as well as the applied fractal dimension in counting their set\cite{I16}. 

\subsection{Study Outline}
This work offers a counterpart existence result for random virtual fractals for given fractal dimension and the expected Lebesgue measure in n-dimensional Euclidean spaces. It also provides answer to above posed questions regarding independency of the cardinality of the set of virtual fractals and their deterministic or random status and the applied fractal dimensions. The work outline is as follows: First, it introduces the fat fractal perculations (FFP) and summarizes some of their topological properties and fractal dimensions. The ancillary proofs are furnished for establishing these properties in the Appendix section. Second, it establishes two key lemmas regarding the cardinality of the power set of the surviving FFP and the case of HDT for random fractals. Finally, using the mentioned two lemmas, it establishes the random counterpart of the GHDT coined as the Second Generalized Hausdorff Dimension Theorem (SGHDT).


\section{Preliminaries}
The reader who has studied random fractals is well-equipped with the following 
definitions and key properties of the fat fractal perculation (FFP). The summary of definitions and some key properties of the Hausdorff dimesnion and the topological dimension (denoted by $\dim_H(.),\dim_{ind}(.),$ respectively) are presented in \cite{I16}. Furthermore, the summary and key properties of the Packing dimension, Assouad dimension and the Boxing dimension (denoted by $\dim_P(.),\dim_A(.),\dim_B(.),$ respectively) are presented in \cite{P1,P2}.    Henceforth, in this paper we consider the n-dimensional unit cube $I_n (n\geq 1)$ in the Euclidean space $\mathbb{R}^n$ with its conventional Euclidean metric and the Lebesgue measure of  $\lambda_n(.).$\par 

The earliest ideas of the FFP root to Mandelbrot's work in 1982 \cite{P3}. The construction process of the FFP is as follows: 
Given a dimension $n\geq 1,$ an index parameter $m\geq 2$ and a non-decreasing sequence $\overrightarrow{p}=\{p_k\}_{k=1}^{\infty}$ in $(0,1].$ Let $C_{n,m,\overrightarrow{p}}(0)$ be the unit cube $I_n.$ Divide $C_{n,m,\overrightarrow{p}}(0)$ into $m^n$ equal closed subcubes, each with length $m^{-1},$ and, then select each subcube independently with probability of $p_1.$ Denote the union of chosen subcubes as $C_{n,m,\overrightarrow{p}}(1).$ Continuing this process, at the stage $k>1,$ divide each subcube of $C_{n,m,\overrightarrow{p}}(k-1)$ into $m^n$ equal closed subcubes, each with length $m^{-k},$ and, then select each subcube independently with probability of $p_k.$ Similar to the case $k=1,$ denote the union of chosen subcubes as $C_{n,m,\overrightarrow{p}}(k).$ The sequence of random closed sets $\{C_{n,m,\overrightarrow{p}}(k)  \}_{k=0}^{\infty}$ is decreasing. 
\begin{defi}
\label{Def1}
The fat fractal perculation(FFP) $C_{n,m,\overrightarrow{p}}$ associated with dimension $n\geq 1,$  index parameter $m\geq 2$ and the non-decreasing sequence $\overrightarrow{p}$ is defined as:
\begin{equation}
\label{eqp1}
C_{n,m,\overrightarrow{p}}=\bigcap_{k=0}^{\infty} C_{n,m,\overrightarrow{p}}(k)
\end{equation}
\end{defi}
\noindent To study FFP in Definition \ref{Def1}, we consider the probability space $(\Omega_n,\mathcal{F}_n,P_p) ,  (n\geq 1)$ with the following characteristics.  Let $\Omega_n=I_n,$ and $\mathcal{F}_n (n\geq 1)$ be its space of compact subsets. We define a natural product probability measure $P_p$ as follows: let $D_{n,m,k}(k\geq 0)$ be the set of all $(m^n)^k$ closed subcubes of $I_n$ each with side length $m^{-k}.$ For any $k\geq 1$ and each $E\in D_{n,m,k}$ define a probaility measure $P_{n,m,\overrightarrow{p}}(E)=\frac {1_{\mathcal{P}(C_{n,m,\overrightarrow{p}}(k))}(E)}{\prod_{l=1}^{k} p_l\times(m^n)^{k}},$ where $\mathcal{P}(.)$ refers to the power set. Then, by the extension theorem \cite{P4} there is a unique measure $P_p$ on $I_n$ such that $P_p(E)=P_{n,m,\overrightarrow{p}}(E),$ for all $E\in D_{n,m,k}, k\geq 1.$ We say a property holds almost surely(a.s) when it holds on a set of full $P_p$ measure. 
\begin{remark}
\label{rem1}
The probability space  $(\Omega_n,\mathcal{F}_n,P_p) ,  (n\geq 1)$ can be easily generalized for non-unit n dimensional cubes. To see this, let  $0\leq a<b,$ and $I_{n,a,b}=[a,b]^n.$ Define $\Omega_{n,a,b}=(b-a)\Omega_n+a, \mathcal{F}_{n,a,b}=\{ (b-a)E+a| E\in \mathcal{F}_n\}, $ and $P_{p,a,b}$ by $P_{p,a,b}(E)=P_p(\frac{E-a}{b-a}), (E\subseteq I_{n,a,b}).$ Then, $(\Omega_{n,a,b},\mathcal{F}_{n,a,b},P_{p,a,b}) $ is a probability space. Trivially,$(\Omega_n,\mathcal{F}_n,P_p)=(\Omega_{n,0,1},\mathcal{F}_{n,0,1},P_{p,0,1}) ,  (n\geq 1)$. In the upcoming results, we will use this probability space whenever it is useful within its context.
\end{remark}

\begin{remark}
We refer to the set of random fractals in $\mathbb{R}^n$ via random fractals charactreized by the probability space $(\Omega_{n,a,b},\mathcal{F}_{n,a,b},P_{p,a,b}) $ where $-\infty<a<b<\infty$ are clear from the context of the discussion.
\end{remark}

\begin{remark}
\label{rem2}
Given  $b>0$ and $n\geq 1.$ Then, the two probability spaces $(\Omega_{n,0,b},\mathcal{F}_{n,0,b},P_{p,0,b}), $ and $(\Omega_{n,b,2b},\mathcal{F}_{n,b,2b},P_{p,b,2b})$ can be considered as two ``subspaces" of the probability space $(\Omega_{n,0,2b},\mathcal{F}_{n,0,2b},P_{p,0,2b}),$ in set theoric context. To see this, it is sufficient to consider that in the first round of above construction process all chosen  subcubes from $\Omega_{n,0,2b}$ to be chosen from  either $\Omega_{n,0,b}$ or $\Omega_{n,b,2b},$ respectively. 
\end{remark}

\begin{remark}
\label{rem3}
Given random fractal $C $  in $\mathbb{R}^n,$ both of its Lebesgue measure $\lambda_n(C)$ and fractal dimesnion $\dim_{F}(C)$ are random variables. Consistent with literature in random fractals we characterize $C$ in terms of its expected Lebesgue measure and its almost sure value of fractal dimension.
\end{remark}

We are interested in the crucial topological and geometrical characteristics of above defined FFPs(See Appendix for some of the proofs). The first question is when the FFP survives the extiction (i.e., $C_{n,m,\overrightarrow{p}}\neq \emptyset$)?The first Theorem discuss this question and shed light on some of its topological properties \cite{P5,P6,P7,P8}:

 \begin{thm}
\label{Theoremp1}
Let $C_{n,m,\overrightarrow{p}}$ be the FFP constructed as in Definition \ref{Def1}. Then, for $\alpha=\liminf\limits_{k\rightarrow \infty}(\prod_{l=1}^{k}p_l)^{\frac{1}{k}}$ and $\beta=\prod_{k=1}^{\infty}p_{k}^{m^{nk}}$we have:\\
(i)  $C_{n,m,\overrightarrow{p}}$ is empty set or non-empty with positive probability whenever $\alpha\leq m^{-n},$ or $\alpha>m^{-n},$ respectively. Also, in the later consition, it is uncountable as well.\\
(ii) $C_{n,m,\overrightarrow{p}}$ has empty interior or non-empty interior whenever $\beta=0$ or $\beta>0,$ respectively. Furthermore,  in the later condition, it equals almost surely to a union of finitely many cubes. 
\end{thm}

Next, the most two prominent quantities essential for our upcoming results are the Fractal dimension and the Lebesgue measure. We refer to Fractal dimension (denoted by $\dim_F(.)$) as one of four dimesnions: the Hausdorff dimension, the Packing dimension, the Assouad dimension or the Box dimension. The following theorem quantifies these key quantities \cite{P1,P6}:
 \begin{thm}
\label{Theoremp2}
Let $C_{n,m,\overrightarrow{p}}$ be the FFP constructed as in Definition \ref{Def1}. Then:\newline\\
(i)\ $E(\lambda_n(C_{n,m,\overrightarrow{p}})) =  \prod_{k=1}^{\infty}p_k,$ \\
(ii)\ $\dim_H(C_{n,m,\overrightarrow{p}})=^{a.s} n+   \log_m(\liminf\limits_{k\rightarrow \infty}(\prod_{l=1}^{k}p_l)^{\frac{1}{k}}),$\\
(iii)\  $\dim_P(C_{n,m,\overrightarrow{p}})=^{a.s}
 \limsup\limits_{k \rightarrow \infty} (
\frac{n+\log_m (\prod_{l=1}^{k+1}p_l)^{\frac{1}{k+1}}}{{1+\frac{1}{n}\log_m(p_{k+1}^{\frac{1}{k+1}})}} ),$\\
(iv)\ $\dim_A(C_{n,m,\overrightarrow{p}})=^{a.s} n+
 \limsup\limits_{t\rightarrow \infty} (\sup\limits_{k\geq 1}(\log_m (\prod_{l=k}^{t+k}p_l)^{\frac{1}{t+1}})),$\\
(v)  $\dim_B^{lower}(C_{n,m,\overrightarrow{p}})=\dim_H(C_{n,m,\overrightarrow{p}})$ and $\dim_B^{upper}(C_{n,m,\overrightarrow{p}})=\dim_P(C_{n,m,\overrightarrow{p}})$.
\end{thm}

An immediate consequence of the Theorem \ref{Theoremp2} is the following:

\begin{cor}
\label{cor1}
The FFP constructed in Definition \ref{Def1} has a positive expected Lebesgue measure if and only if its Hausdorff dimension equals to $n,$ almost surely. 
\end{cor}

We finish this section with some prominent examples useful for the key results in the subsequent section.

\begin{examp}
\label{example1}
Let the sequence of probabilities $\{ p_k\}_{k=1}^{\infty}$ be defined by $p_k=p^{a_k}$ where $m^{-n}<p<1$ is fixed and $a_k>0 (k\geq 1).$ Then, by above Theorem, it follows that:\newline\\
 (i)\ $E(\lambda_n(C_{n,m,\overrightarrow{p}})) = p^{\sum_{k=1}^{\infty} a_k},$\\
(ii) $\dim_H(C_{n,m,\overrightarrow{p}})=^{a.s} n+(\liminf_{k\rightarrow \infty}\frac{\sum_{l=1}^{k}a_l}{k})\times \log_m(p),$\\
(iii) $\dim_P(C_{n,m,\overrightarrow{p}})=^{a.s}  \limsup\limits_{k \rightarrow \infty}(\frac{n+\frac{\sum_{l=1}^{k+1}a_l}{k+1}\times \log_m(p)} {1+\frac{a_{k+1}}{k+1}\frac{1}{n}\log_m(p)}), $\\
(iv) $\dim_A(C_{n,m,\overrightarrow{p}})=^{a.s} n+ \limsup\limits_{t\rightarrow \infty} (\sup\limits_{k\geq 1} \frac{\sum_{l=k}^{k+t}a_l}{t+1}\times \log_m(p)).$\\
We will use this special example for the proof of Lemma \ref{Lemma3.2} in subsequent section.
\end{examp}

\begin{examp}
\label{example2}
Let in the Example \ref{example1}, $a_k=1 (k\geq 1).$ Then, the obtained random fractals are referred to Mandelbrot fractal perculations (MFP). In particular,   for $p>m^{-n}$  the MFP survives extinction, $E(\lambda_n(C_{n,m,\overrightarrow{p}}))=0,$ and $\dim_H(C_{n,m,\overrightarrow{p}})=\dim_P(C_{n,m,\overrightarrow{p}})=\dim_A(C_{n,m,\overrightarrow{p}})=\dim_B(C_{n,m,\overrightarrow{p}})=n+\log_m(p),$ almost surely.  We write for simplicity $\dim_F(C_{n,m,\overrightarrow{p}})=n+\log_m(p),$ almost surely.
\end{examp}

 
\section{Main Results}

We generalize the existential Generalized Hausdorff Dimension Theorem(GHDT) from deterministic fractals and one fractal dimension to random fractals and four fractal dimensions. The existential cardinality is aleph-two, as well. The construction process is accomplished
in threer stages: (i) calculting the cardinality of the power set of a surviving random fractal; (ii)  showing the existence of continuum of random fractals with a plausible fractal dimension and expected Lebesgue measure in $n$ dimensional Euclidean space $\mathbb{R}^n(n \geq 1),$ and, (iii) generalizing the result in the second stage to the cardinal of aleph-two. We begin with the following Lemma of cardinality calculation which plays a key role in the SGHDT:
 
\begin{lem}
\label{Lemma3.1}
Given a random fractal $C_{n,m,\overrightarrow{p}}$ with almost surely positive Hausdorff dimension in the unit cube $I^n (n \geq 1).$ Then, the cardinality of its power set equals to aleph-two.
\end{lem} 

\begin{proof}
Any countable event has Hausdorff dimension of zero, almost surely\cite{R1}. Hence, any event with almost surely positive Hausdorff dimension is uncountable. Accordingly, its power set has cardinality of aleph-two given generalized continuum hypothesis (GCH).
\end{proof} 
 
Next, equipped with prelimiary results in the previous section, we establish our first major result on existence of random fractals for given fractal dimension and Lebesgue measure. This result on its own is a direct generalization of the Hausdorff Dimension Theorem(HDT)\cite{I13} from the case of deterministic thin fractals to the case of random fat fractals, and, from one fractal dimension to four fractal dimensions:
 
\begin{lem}
\label{Lemma3.2} 
For any real $r>0$ and $l\geq 0$ there are continuum random fractals  with the Hausdorff dimension $r.1_{\{0\}}(l) +n.1_{(0,\infty)}(l),$ almost surely and the expected Lebesgue measure $l$   in $\mathbb{R}^n (\lceil r\rceil \leq n).$
\end{lem}

\begin{proof}
We consider two sets of random fractals  $C_{n,m,\overrightarrow{p}_{a}}$ where each has cardinality of continuum and its members survive extinction.  Table\ref{Table1} summarizes the key features of these two sets via applying Example \ref{example1}:

\begin{table}[H]
\caption{Two sets with cardinality of continuum of random fractals  in $\mathbb{R}^{n} (n\geq 1)$\label{Table1} }
\begin{threeparttable}
\begin{tabular}{ l|cccc } 
\hline
$C_{n,m,\overrightarrow{p}_{a}}$ & $\overrightarrow{p}_{a}=\{p_k(a)\}_{k=1}^{\infty} $ & $E(\lambda_n(.))$ & $\dim_{H}(.)$ & Constraints\\
\hline
$\#1$ &
$p_k(a)= p^{a 1_{\lbrace 1\rbrace}(k)+1_{[2,+\infty)}(k)}$
&
$0$&
$n+\log_m(p)$& 
$a \geq 1$\\
$\#2$ &
$p_k(a)= p^{(a^{k-1}-a^{k})} $&
$p$ &
$n$ &
$0<a<1$\\
\hline
\end{tabular}
\begin{tablenotes}
\item[] \small{Notes: $m^{-n}<p<1$ is taken as fixed. Values for the Hausdorff dimension are taken almost surely. Each constraint has the cardinality of continuum for the parameter $a.$ } 
\end{tablenotes}
\end{threeparttable}
\end{table}

Now, given plausible values for the Hausdorff dimension and the Lebesgue measure we have three major cases:\newline\\
(i) $0<r \notin \mathbb{N}, l=0:$\newline
Here, for $n\geq \lfloor r\rfloor+1=\lceil r\rceil,$ and $p=m^{r-n}$ for the sets $\#1$ in Table\ref{Table1} it follows that $\dim_{H}(C_{n,m,\overrightarrow{p}_{a}})=r,$ almost surely.\newline\\
(ii) $0<r \in \mathbb{N}, l=0:$\newline
Here, for $n\geq r= \lceil r\rceil$, by part (i) for the sequence $\{r-2^{-k} \}_{k=1}^{\infty}$ there is corresponding sequence $\{ C_{n,m,\overrightarrow{p}_{a}}^{(k)} \}_{k=1}^{\infty}$ such that $\dim_{H}(C_{n,m,\overrightarrow{p}_{a}}^{(k)})=r-2^{-k}, (k\geq 1),$ and $E(\lambda_n(C_{n,m,\overrightarrow{p}_{a}}^{(k)}))=0, (k\geq 1).$  Next, take $C=\cup_{k=1}^{\infty}C_{n,m,\overrightarrow{p}_{a}}^{(k)},$ then $\dim_{H}(C)=\sup_{1\leq k\leq \infty} (\dim_{H}(C_{n,m,\overrightarrow{p}_{a}}^{(k)}))=r,$ and $E(\lambda_n(C))=0.$ Finally, the assertion follows from the uncountability of the sets in part(i).\newline\\
(iii)$0<r \in \mathbb{N}, l>0:$\newline
Here, for $n\geq r= \lceil r\rceil$, and $p=\frac{l}{\lfloor l\rfloor+1}$ for the sets $\#2$ in Table\ref{Table1} it follows that 
$\dim_{H}((\lfloor l\rfloor+1)^{\frac{1}{n}}*C_{n,m,\overrightarrow{p}_{a}})=r,$ almost surely, and $E(\lambda_n((\lfloor l\rfloor+1)^{\frac{1}{n}}*C_{n,m,\overrightarrow{p}_{a}}))=l.$  \newline\\
Finally, we complete the proof by considering the fact that counstructed random fractals in cases (i),(ii) belong to the standard probability space $(\Omega_{n,0,1},\mathcal{F}_{n,0,1},P_{p,0,1})$ while those in case(iii) belong to generalized probability space $(\Omega_{n,a,b},\mathcal{F}_{n,a,b},P_{p,a,b}):\ \ a=0,b=(\lfloor l\rfloor+1)^{\frac{1}{n}}.$
\end{proof}

\begin{remark}
\label{rem3}
We can prove Lemma \ref{Lemma3.2} in an alternative method. In this method, set $a=1$ and $a=\frac{1}{2}$ for the sets of random fractals in Table \ref{Table1}, respectively. Then, consider the set $\mathcal{G}_n$ with cardinality of continuum of random fractals defined by $\mathcal{G}_n=\{\cup_{k\in I} (2^{-k}C_{n,m,\overrightarrow{p}_{a}}+(1-2^{k-1})1_n)| I\subseteq \mathbb{N}\  \text{is\ infinite} , 1_n=\sum_{i=1}^{n}e_i\  (n\geq 1)\}.$ Finally, the new proof is completed by considering a similar argument to the outlined three major cases mentioned in the first proof of the Lemma \ref{Lemma3.2}.
\end{remark}

The constructed random fractals in Lemma \ref{Lemma3.2} share one key feature with the deterministic ones in \cite{I16}: there is a continuum of the random fractals with the same fractal dimension and the same expected Lebesgue measure in
the Euclidean spaces  $\mathbb{R}^{n+1}-\mathbb{R}^n (n\geq 1)$   as those in $\mathbb{R}^n (n\geq 1).$ This is a direct consequence of considering $\Omega_{n,a,b}$  isomorphic to the subspace of $\Omega_{n,a,b} \times \{0\}  \subseteq \Omega_{n+1,a,b},$ for the contextual $a<b.$ Furthermore, as in the deterministic case, the result in Lemma \ref{Lemma3.2} is limited on providing the highest possible cardinal number of aleph-two. Our final result provides them:

\begin{thm}
\label{Theorem1}
\textbf{(The Second Generalized Hausdorff Dimension Theorem)} For any real $r > 0$ and $l\geq 0,$ there are aleph-two random fractals with the Hausdorff dimension $r.1_{\{0\}}(l)+n.1_{(0,\infty)}(l)$ almost surely,  and expected Lebesgue measure $l$  in $\mathbb{R}^{n}$ where $(\lceil r\rceil \leq n)$.
\end{thm}

\begin{proof}
Let $r>0, l\geq 0,$ and fix $n\geq \lceil r\rceil.$ Then, by two applications of the Lemma \ref{Lemma3.2} we have the following families of random fractals where each has the minimum cardinality of continuum:

\begin{eqnarray}
\label{eqr4}
\mathcal{G}_1(r,l)&=& \Big{\{} G\in \mathcal{F}_{n,(\lfloor l\rfloor+1)^{\frac{1}{n}},2(\lfloor l\rfloor+1)^{\frac{1}{n}}}  \Big{|} \begin{matrix}
 \dim_{H}(G)&\leq ^{a.s}\frac{r}{2} \\
E(\lambda_n(G))&=0
\end{matrix}  \Big{\}}\nonumber\\
\mathcal{G}_2(r,l)&=&\Big{\{} G\in \mathcal{F}_{n,0,(\lfloor l\rfloor+1)^{\frac{1}{n}}}\Big{|}\begin{matrix}
\dim_{H}(G)&=^{a.s}r.1_{\{0\}}(l)+n.1_{(0,\infty)}(l)   \\
 E(\lambda_n(G))&=l  
\end{matrix}  \Big{\}}.
\end{eqnarray}
Next, by Remark \ref{rem2}, both family of events in equation \ref{eqr4} belong to the probability space $(\Omega_{n,0,2b},\mathcal{F}_{n,0,2b},P_{p,0,2b}) $ where $b=(\lfloor l\rfloor+1)^{\frac{1}{n}}$ . Also, an application of Lemma \ref{Lemma3.1} showes that the first family $\mathcal{G}_1(r,l)$ in equation (\ref{eqr4}) has cardinalty of aleph-two. Consequently,   the following family of random fractals has cardinality of aleph-two: 
\begin{equation}
\label{eqr}
\mathcal{G}(r,l)=\{G_1\cup G_2| G_1\in \mathcal{G}_1(r,l), G_2\in \mathcal{G}_2(r,l)\}. 
\end{equation}
\noindent Furthermore, let $G\in\mathcal{G}(r,l).$ Then, by definition, there are $G_i\in\mathcal{G}_i(i=1,2)$ such that: $G=G_1\cup G_2.$  Now, by conventional properties of the Hausdorff dimension : 
\begin{eqnarray}
\dim_{H}(G)&=&\max(\dim_{H}(G_1),\dim_{H}(G_2))\nonumber\\
&=&\dim_{H}(G_2)=^{a.s}r.1_{\{0\}}(l)+n.1_{(0,\infty)}(l).
\end{eqnarray}
Also, we have:
\begin{eqnarray}
E(\lambda_n(G))&=&E(\lambda_n(G_1\cup G_2))\geq  E(\lambda_n(G_2))=l\nonumber \\
&=& E(\lambda_n(G_1))+ E(\lambda_n(G_2))\geq E(\lambda_n(G_1\cup G_2))=E(\lambda_n(G))\nonumber
\end{eqnarray}
yielding:
\begin{eqnarray}
E(\lambda_n(G))&=& l.
\end{eqnarray}
This completes the proof.
\end{proof}

\begin{remark}
Using Theorem\ref{Theoremp2} and Example\ref{example1} the results in Lemma \ref{Lemma3.2} and, consequently, Theorem\ref{Theorem1} hold for the other three  fractal dimensions as well.
\end{remark}

As in the deterministic case, the assertion in Theorem \ref{Theorem1} and its proof methodology have two immediate important corollaries: First, the cardinality of the set of random fractals in $\mathbb{R}^n$ is aleph-two.  Second, the cardinality of the set of random non-fractals in $\mathbb{R}^n$ is aleph-two as well. Both results are proved similar to the determinist case \cite{I16} and yield  from two applications of Cantor–Schroder–Bernstein theorem\cite{R2}, respectively.
\section{Discussion}
This work presented an existence theorem for random fractals of a given Hausdorff dimension
and a Lebesgue measure with the highest possible cardinal number of aleph-two. In addition, it generalized the former detrministic existence theorem in terms of the  randomness and the number of fractal dimensions. To compare the generalization process from HDT to GHDT and from GHDT to SGHDT we summarize these results as follows:

\begin{tabular}{ p{1.25cm}  p{13.5cm} }
\textbf{HDT}& For any real $r > 0,$
there is a continuum of thin deterministic fractals with Hausdorff dimension $r$ in n-dimensional
Euclidean space $(\lceil r\rceil \leq n)$.\\
\textbf{GHDT} &  For any real $r > 0$ and $l\geq 0,$ there are aleph-two deterministic fractals with the Hausdorff dimension $r.1_{\{0\}}(l)+n.1_{(0,\infty)}(l)$  and  Lebesgue measure $l$ in $\mathbb{R}^{n}$ where $(\lceil r\rceil \leq n)$.   \\
\textbf{SGHDT} &  For any real $r > 0$ and $l\geq 0,$ there are aleph-two random fractals with the fractal dimension $r.1_{\{0\}}(l)+n.1_{(0,\infty)}(l)$ almost surely,  and expected Lebesgue measure $l$  in $\mathbb{R}^{n}$ where $(\lceil r\rceil \leq n)$. Here, the fractal dimension is one of four dimensions: Hausdorff dimension, Packing dimension, Assouad dimension, Box dimension.
\end{tabular}

Figure.1 presents the generalization process upon comparing the above three statements.
 
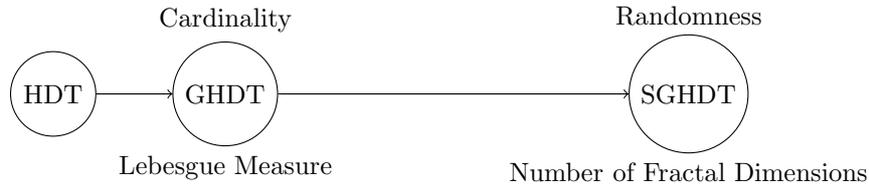
\begin{figure} [H]
\begin{center}
\begin{tikzpicture}
\draw (1,1) node[anchor=  east,circle,
draw](1){HDT} -- (2,1) node[anchor=  west,
circle,draw,
label=above:Cardinality,label=below:
Lebesgue Measure](2){GHDT};
\draw (3.45,1) --(8,1) node[anchor=  west,
circle,draw,label=above:Randomness,label=below:
Number of Fractal Dimensions](3){SGHDT};
\draw [->] (1) -- (2);
\draw [->] (2) -- (3);
\end{tikzpicture}
\end{center} 
\caption{The generalization process of the Hausdorff Dimension Theorem(HDT) in terms of cardinality, Lebesgue measure, randomness, and number of fractal dimensions} \label{fig:M1}
\end{figure}

This work’s contributions to the fractal geometry literature are in four perspectives: First, as in deterministic case, it highlights the advantages of random Cantor fractals on other conventional random fractals on establishing existential results for random fractals. Second, as in deterministic case, the main result is equipped with constructive proof rather than pure existential proof. Third, it presents two more sets with cardinality of aleph-two. Finally, it partially answers open problems \#1 and \#3 in \cite{I16}. For the case of open problem \#1, the dimension of the set of distinctive random fractals in $\mathbb{R}^n$ is the same for four fractal dimensions considered in this work. Moreover, for the case of open problem \#3, the cardinality of the set of all distinctive fractals in $\mathbb{R}^n$ is independent from their determinitic or random nature.  The work limitation -as in the deterministic case- is its limited generalizability for more generalized and abstract fractal structures and their fractal dimensions\cite{D1}. 

\section{Conclusion}  
This work has presented another partial solution to the problem of retrospective existence of any set of fractals for a given fractal dimension and the Lebesgue measure for the case of random fractals. It presented the case for random fractals with cardinality of aleph-two, and expanded the former deterministic result in terms of randomness and the number of involved fractals dimensions. Finally, it cleared the way for working through the problem for the  case of  more generalized  abstract fractal structure and dimension.

\subsection*{Abbreviations}
\noindent The following abbreviations are used in this manuscript:\\

\noindent 
\begin{tabular}{@{}ll}
\text{a.s} & \text{Almost Surely }\\
\text{FFP} & \text{Fat Fractal Perculation }\\
\text{HDT} & \text{Hausdorff Dimension Theorem }\\
\text{GCH} & \text{generalized continuum hypothesis}\\ 
\text{GHDT} & \text{Generalized Hausdorff Dimension Theorem}\\
\text{MFP}&\text{Mandelbrot Fractal Perculation}\\
\text{SGHDT} & \text{Second Generalized Hausdorff Dimension Theorem}
\end{tabular}

\appendix
\section{Some Ancillary Proofs}
\subsection{Proof for Theorem \ref{Theoremp1}. (i)}
 This is immediate consequence of Theorem \ref{Theoremp2} (ii) and properties of the Hausdorff dimension \cite{P5}. 
\subsection{Proof for Theorem \ref{Theoremp2}. (i)}
 Let $k\geq 1$ be fixed. Then, by conditioning and mathematical induction the number of subcubes $X_k$ in  $C_{n,m,\overrightarrow{p}}(k)$ satisfies $X_k\sim $ \\ $ Binomial((m^n)^k, \prod_{l=1}^{k}p_l).$ But, each subcube has length of $m^{-k}$ and the Lebesgue measure of $(m^n)^{-k}.$ Hence, $E(\lambda_n(C_{n,m,\overrightarrow{p}}(k)))=E(X_k.(m^n)^{-k})= \prod_{l=1}^{k}p_l , (k\geq 1).$ Now, the result follows from an application of the Fubini's Theorem and letting $k\rightarrow +\infty.$   
\subsection{Proof for Theorem \ref{Theoremp2}. (ii)}
 Let $d=n,$ $M_k=m$ and $N_k=p_k*m^n(k\geq 1)$ as in  \cite{P1,P2}. Then:  $\dim_H(C_{n,m,\overrightarrow{p}})=\liminf\limits_{k\rightarrow \infty} \frac{\log(\prod_{l=1}^{k}N_l)}{-\log(\prod_{l=1}^{k}M_l^{-1})}
=\liminf\limits_{k\rightarrow \infty}  \frac{\log(\prod_{l=1}^{k}p_l (m^k)^n)}{ \log(m^k)}
=\liminf\limits_{k\rightarrow \infty} (n+ \frac{\log(\prod_{l=1}^{k}p_l )}{ \log(m^k)})
=\liminf\limits_{k\rightarrow \infty} (n+ \frac{\log((\prod_{l=1}^{k}p_l )^{\frac{1}{k}})}{ \log(m)})
=n+\liminf\limits_{k\rightarrow \infty}  \log_m((\prod_{l=1}^{k}p_l)^{\frac{1}{k}}) 
=  n+   \log_m(\liminf\limits_{k\rightarrow \infty}(\prod_{l=1}^{k}p_l)^{\frac{1}{k}}).
$
\subsection{Proof for Theorem \ref{Theoremp2}. (iii)}
 Similar to the case for Hausdorff dimension, let $d=n,$ $M_k=m$ and $N_k=p_k*m^n(k\geq 1)$ as in  \cite{P1,P2}. Then:  $\dim_P(C_{n,m,\overrightarrow{p}})=\limsup\limits_{k\rightarrow \infty} \frac{\log(\prod_{l=1}^{k+1}N_l)}{-\log(\prod_{l=1}^{k}M_l^{-1})+\frac{1}{n}\log(N_{k+1})}=\limsup\limits_{k\rightarrow \infty}(\frac{\log(\prod_{l=1}^{k+1}p_l(m^{n(k+1)}))}
{\log(m^k)+\frac{1}{n}\log(p_{k+1}m^n)})=$  $\limsup\limits_{k\rightarrow \infty} (\frac{n(k+1)\log(m)+\log(\prod_{l=1}^{k+1}p_l)}{(k+1)\log(m)+\frac{1}{n}\log(p_{k+1})})=\limsup\limits_{k\rightarrow \infty} (\frac{n+\frac{\log(\prod_{l=1}^{k+1}p_l)}{(k+1)\log(m)}}{1+\frac{\log(p_{k+1})}{n(k+1)\log(m)}})=$\\$\limsup\limits_{k\rightarrow \infty} (\frac{n+ \log_{m}(\prod_{l=1}^{k+1}p_l)^{\frac{1}{k+1}}}{1+\frac{1}{n}\log_{m}(p_{k+1}^{\frac{1}{k+1}})} ) .$  

\subsection{Proof for Theorem \ref{Theoremp2}. (iv)}
 Similar to the case for Hausdorff dimension, let $d=n,$ $M_k=m$ and $N_k=p_k*m^n(k\geq 1)$ as in  \cite{P1,P2}. Then:  $\dim_A(C_{n,m,\overrightarrow{p}})= \limsup\limits_{t\rightarrow \infty} (\sup\limits_{k\geq 1}(\frac{\prod_{l=k}^{k+t}N_l}{-\log(\prod_{l=k}^{k+t}M_l^{-1})} ))=\limsup\limits_{t\rightarrow \infty} (\sup\limits_{k\geq 1}\frac{\log(\prod_{l=k}^{k+t}(p_l.m^n))}{\log(\prod_{l=k}^{k+t}m)})=$  $ \limsup\limits_{t\rightarrow \infty}(\sup\limits_{k\geq 1}\frac{\log(\prod_{l=k}^{k+t}p_l.(m^{n(t+1)}))}{\log(m^{t+1})})=\limsup\limits_{t\rightarrow \infty}(\sup\limits_{k\geq 1}\frac{\log(\prod_{l=k}^{k+t}p_l)+n \log(m^{t+1})}{\log(m^{t+1})} )=$ \\ $\limsup\limits_{t\rightarrow \infty}(\sup\limits_{k\geq 1}\frac{\log(\prod_{l=k}^{k+t}p_l)} {(t+1)\log(m) }+n)=\limsup\limits_{t\rightarrow \infty}(\sup\limits_{k\geq 1}\log_{m}((\prod_{l=k}^{k+t}p_l)^{\frac{1}{t+1}}) +n)= $  $ n+
 \limsup\limits_{t\rightarrow \infty} (\sup\limits_{k\geq 1}(\log_m (\prod_{l=k}^{t+k}p_l)^{\frac{1}{t+1}})).$ 

\subsection{Proof for Theorem \ref{Theoremp2}. (v)}
 This is direct consequence from  \cite{P1}.

\subsection{Proof for Corollary \ref{cor1} }
 It is sufficient to use the representation $l=\prod_{k=1}^{\infty}p_k.$ First, assume $l>0.$ Then, using the Calculus Theorem on limit of composition of continuous functions and the representation in part (ii)  the sufficiency is proved. Second, assume $l=0.$ Then, using $\epsilon-\delta$ definition for the sequences and taking $\epsilon=\frac{1}{2},$ the necessity is proved.

\end{document}